\pdfoutput=1
\RequirePackage{ifpdf}
\ifpdf 
\documentclass[pdftex]{sigma}
\else
\documentclass{sigma}
\fi

\begin{document}
\allowdisplaybreaks

\newcommand{\arXivNumber}{2002.03439}

\renewcommand{\thefootnote}{}

\renewcommand{\PaperNumber}{073}

\FirstPageHeading

\ShortArticleName{Nonstandard Quantum Complex Projective Line}

\ArticleName{Nonstandard Quantum Complex Projective Line\footnote{This paper is a~contribution to the Special Issue on Noncommutative Manifolds and their Symmetries in honour of Giovanni Landi. The full collection is available at \href{https://www.emis.de/journals/SIGMA/Landi.html}{https://www.emis.de/journals/SIGMA/Landi.html}}}

\Author{Nicola CICCOLI~$^\dag$ and Albert Jeu-Liang SHEU~$^\ddag$}

\AuthorNameForHeading{N.~Ciccoli and A.J.-L.~Sheu}

\Address{$^\dag$~Dipartimento di Matematica e Informatica, University of Perugia, Italy}
\EmailD{\href{mailto:nicola.ciccoli@unipg.it}{nicola.ciccoli@unipg.it}}

\Address{$^\ddag$~Department of Mathematics, University of Kansas, Lawrence, KS 66045, USA}
\EmailD{\href{mailto:asheu@ku.edu}{asheu@ku.edu}}

\ArticleDates{Received March 06, 2020, in final form July 24, 2020; Published online August 03, 2020}

\Abstract{In our attempt to explore how the quantum nonstandard complex projective spaces $\mathbb{C}P_{q,c}^{n}$ studied by Korogodsky, Vaksman, Dijkhuizen, and Noumi are related to those arising from the geometrically constructed Bohr--Sommerfeld groupoids by Bonechi, Ciccoli, Qiu, Staffolani, and Tarlini, we were led to establish the known identification of $C\big(\mathbb{C}P_{q,c}^{1}\big) $ with the pull-back of two copies of the Toeplitz $C^*$-algebra along the symbol map in a~more direct way via an operator theoretic analysis, which also provides some interesting non-obvious details, such as a~prominent generator of $C\big( \mathbb{C}P_{q,c}^{1}\big) $ being a~concrete weighted double shift.}


\Classification{58B32; 46L85}

\renewcommand{\thefootnote}{\arabic{footnote}}
\setcounter{footnote}{0}

\section{Introduction}

In \cite{Sh:qcp}, the $C^*$-algebra $C\big( \mathbb{C}P_{q,c}^{n}\big) $ of
nonstandard quantum complex projective spaces studied by Korogodsky and
Vaksman \cite{KoVa} and Dijkhuizen and Noumi~\cite{DiNo} is embedded in a
concrete groupoid $C^*$-algebra, and then shown to have $C\big( \mathbb{S}_{q}^{2n-1}\big) $ as a quotient algebra, which reflects the geometric observation~\cite{Sh:cp} that the nonstandard ${\rm SU}( n+1)
$-covariant Poisson complex projective space $\mathbb{C}P^{n}$ contains a copy
of the standard Poisson sphere $\mathbb{S}^{2n-1}$.

Although the work in \cite{Sh:qcp} involves realizing $C\big( \mathbb{C}P_{q,c}^{n}\big) $ as part of a concrete groupoid $C^*$-algebra in order to analyze the algebra structure and extract useful information, it is not clear
whether one can actually realize $C\big( \mathbb{C}P_{q,c}^{n}\big) $ as
a groupoid $C^*$-algebra itself. However from a~purely differential geometric
consideration, an elegant program of constructing some quantum homogeneous
spaces as the groupoid $C^*$-algebras of geometrically constructed
Bohr--Sommerfeld groupoids is later successfully developed by Bonechi, Ciccoli,
Qiu, Staffolani, Tarlini~\cite{BCQT, BCST}. Naturally, it is of great interest
to decide whether the quantum complex projective space arising from this new
program is indeed the same as the known version of $C\big( \mathbb{C}P_{q,c}^{n}\big) $. Indeed they are the same for the case of $n=1$ because the underlying groupoids are shown to be isomorphic in Proposition~7.2 of~\cite{BCQT}. But the higher dimensional cases are far from being settled.

It is hoped that by analyzing more carefully the embedding of $C\big(\mathbb{C}P_{q,c}^{n}\big) $ in a concrete groupoid $C^*$-algebra found in~\cite{Sh:qcp} via a representation theoretic approach, one can see some direct
connection with the geometrically constructed Bohr--Sommerfeld groupoid and
then possibly find a way to identify these two different versions of quantum
complex projective spaces. While attempting this approach, we come to
recognize the need of a more direct understanding of the algebra structure of
$C\big( \mathbb{C}P_{q,c}^{n}\big) $ based on some known representations
of the ambient algebra $C( {\rm SU}_{q}( n+1)) $.

In particular, for $n=1$, we want to directly derive the algebra structure of
$C\big( \mathbb{C}P_{q,c}^{1}\big) $ from the basic representations of
$C( {\rm SU}_{q}(2)) $, instead of via identifying
$C\big( \mathbb{C}P_{q,c}^{1}\big) $ with the algebra $C\big(\mathbb{S}_{\mu c}^{2}\big) $ of the Podle\'{s} quantum 2-sphere \cite{Po}
as indicated in~\cite{DiNo, KoVa}. In this note, we show how to accomplish it.
Along the way, our detailed analysis reveals some nontrivial hidden
structures, for example, a distinguished generator $x_{1}^{\ast}x_{2}$ of
$C\big( \mathbb{C}P_{q,c}^{1}\big) $ is a weighted double shift (on a~core Hilbert space that determines the $C^*$-algebra structure of $C\big(\mathbb{C}P_{q,c}^{1}\big) $) with respect to an orthonormal basis, and its weights are determined by a concrete formula.

\section[Nonstandard quantum $\mathbb{C}P_{q,c}^{1}$]{Nonstandard quantum $\boldsymbol{\mathbb{C}P_{q,c}^{1}}$}

We recall the description of $C\big( \mathbb{C}P_{q,c}^{n}\big) $ with
$c\in ( 0,\infty ) $ and $q\in ( 1,\infty ) $ obtained
by Dijkhuizen and Noumi~\cite{DiNo} as
\[
C\big(\mathbb{C}P_{q,c}^{n}\big)\cong C^{\ast}(\{x_{i}^{\ast}x_{j}\,|\,1\leq i,j\leq
n+1\})\subset C ( {\rm SU}_{q}(n+1) ),
\]
where
\[
x_{i}=\sqrt{c}u_{1,i}+u_{n+1,i}
\]
for the standard generators $\{ u_{i,j}\} _{i,j=1}^{n+1}$ of $C( {\rm SU}_{q}(n+1))$.

In this paper, we focus on the case of $n=1$, with the goal to directly show
that $C\big( \mathbb{C}P_{q,c}^{1}\big) $ is the pullback $\mathcal{T}\oplus_{C( \mathbb{T}) }\mathcal{T}$ of two copies of the
standard symbol map
\[
\sigma\colon \ \mathcal{T}\rightarrow C ( \mathbb{T} )
\]
for the Toeplitz algebra $\mathcal{T}$ that is the $C^*$-algebra generated by the
(forward) unilateral shift $\mathcal{S}$ on $\ell^{2}( \mathbb{Z}_{\geq
}) $, with $\ker( \sigma) =\mathcal{K}\big( \ell
^{2} ( \mathbb{Z}_{\geq} )\big) $ the ideal of all compact operators.

For $C ( {\rm SU}_{q}(2) ) $, consider the known faithful
representation $\pi$ of $C ( {\rm SU}_{q}(2) ) $
determined by
\[
\pi(u) \equiv\left(
\begin{matrix}
\pi ( u_{11} ) & \pi ( u_{12} ) \\
\pi ( u_{21} ) & \pi ( u_{22} )
\end{matrix}
\right) :=\left\{ \left(
\begin{matrix}
t_{1}\alpha & -q^{-1}t_{1}\gamma\\
t_{2}\gamma & t_{2}\alpha^{\ast}%
\end{matrix}
\right) \right\} _{t_{2}=\overline{t_{1}}\in\mathbb{T}}%
\]
as a $\mathbb{T}$-family of representations of $C( {\rm SU}_{q}(2) ) $ on $\ell^{2} ( \mathbb{Z}_{\geq}) $ with
parameter $t_{1}\equiv\overline{t_{2}}\in\mathbb{T}$, where
\[
\alpha=\left(
\begin{matrix}
0 & \sqrt{1-q^{-2}} & 0 & & \\
0 & 0 & \sqrt{1-q^{-4}} & 0 & \\
0 & 0 & 0 & \sqrt{1-q^{-6}} & \ddots\\
& 0 & 0 & 0 & \ddots\\
& & \ddots & \ddots & \ddots
\end{matrix}
\right) \in\mathcal{B}\big( \ell^{2}( \mathbb{Z}_{\geq})\big)
\]
and
\[
\gamma=\left(
\begin{matrix}
1 & 0 & 0 & & \\
0 & q^{-1} & 0 & 0 & \\
0 & 0 & q^{-2} & 0 & \ddots\\
& 0 & 0 & q^{-3} & \ddots\\
& & \ddots & \ddots & \ddots
\end{matrix}
\right) \in\mathcal{B}\big( \ell^{2} ( \mathbb{Z}_{\geq} )
\big) \ \text{\ \ self-adjoint}
\]
satisfying
\[
\alpha^{\ast}\alpha+\gamma\gamma^{\ast}\equiv\alpha^{\ast}\alpha+\gamma
^{2}=I=\alpha\alpha^{\ast}+q^{-2}\gamma^{2}\equiv\alpha\alpha^{\ast}%
+q^{-2}\gamma^{\ast}\gamma
\]
and
\[
\gamma\alpha^{\ast}-q^{-1}\alpha^{\ast}\gamma=\alpha\gamma^{\ast}-q^{-1}%
\gamma^{\ast}\alpha=0\equiv\alpha\gamma-q^{-1}\gamma\alpha=\gamma^{\ast}%
\alpha^{\ast}-q^{-1}\alpha^{\ast}\gamma^{\ast}%
\]
which ensure the required condition $\pi(u) \pi(u)
^{\ast}=I=\pi(u) ^{\ast}\pi(u) $.

In this paper, we identify every element of $C( {\rm SU}_{q}(2)
) \supset C ( \mathbb{C}P_{q,c}^{1} ) $ with a $\mathbb{T}$-family of operators on $\ell^{2} ( \mathbb{Z}_{\geq} ) $ via this
faithful representation $\pi$, and we analyze such a $\mathbb{T}$-family of
operators pointwise at each fixed $t_{1}\in\mathbb{T}$.

More explicitly, the generators $x_{1}^{\ast}x_{2},x_{1}^{\ast}x_{1},x_{2}^{\ast}x_{2}$ of $C\big( \mathbb{C}P_{q,c}^{1}\big) $ are
$\mathbb{T}$-families of operators with
\[
x_{1}:=\sqrt{c}t_{1}\alpha+t_{2}\gamma=\left(
\begin{matrix}
t_{2} & \sqrt{c}t_{1}\sqrt{1-q^{-2}} & 0 & & \\
0 & t_{2}q^{-1} & \sqrt{c}t_{1}\sqrt{1-q^{-4}} & 0 & \\
0 & 0 & t_{2}q^{-2} & \sqrt{c}t_{1}\sqrt{1-q^{-6}} & \ddots\\
& 0 & 0 & t_{2}q^{-3} & \ddots\\
& & \ddots & \ddots & \ddots
\end{matrix}
\right)
\]
and
\[
x_{2}:=-q^{-1}\sqrt{c}t_{1}\gamma+t_{2}\alpha^{\ast}=\left(
\begin{matrix}
-q^{-1}\sqrt{c}t_{1} & 0 & 0 & & \\
t_{2}\sqrt{1-q^{-2}} & -q^{-2}\sqrt{c}t_{1} & 0 & 0 & \\
0 & t_{2}\sqrt{1-q^{-4}} & -q^{-3}\sqrt{c}t_{1} & 0 & \ddots\\
& 0 & t_{2}\sqrt{1-q^{-6}} & -q^{-4}\sqrt{c}t_{1} & \ddots\\
& & \ddots & \ddots & \ddots
\end{matrix}
\right) .
\]

At any fixed $t_{1}\in\mathbb{T}$, it is easy to see that $\ker (x_{2}) =0$ since
\begin{gather*}
0=x_{2}\left( \sum_{n=0}^{\infty}z_{n}e_{n}\right) =-q^{-1}\sqrt{c}
z_{0}e_{0}+\big( \sqrt{1-q^{-2}}z_{0}-q^{-2}\sqrt{c}z_{1}\big)
e_{1}\\
\hphantom{0=x_{2}\left( \sum_{n=0}^{\infty}z_{n}e_{n}\right) =}{} +\big( \sqrt{1-q^{-4}}z_{1}-q^{-3}\sqrt{c}z_{2}\big) e_{2}+\cdots
\end{gather*}
implies $z_{0}=z_{1}=z_{2}=\cdots=0$, and $\dim( \operatorname{coker}(
x_{2}) ) =1$ since $x_{2}\equiv t_{2}\alpha^{\ast}\equiv
t_{2}\mathcal{S}$ modulo $\mathcal{K}$ is a Fredholm operator of index $-1$.
On the other hand, $\dim( \ker( x_{1}) ) =1$ by
solving $0=x_{1}\Big( \sum\limits_{n=0}^{\infty}z_{n}e_{n}\Big) $ to get that if
$z_{0}=1$, then
\[
z_{n}=( -1) ^{n}t_{2}^{n}q^{\frac{-n ( n-1 ) }{2}%
}\sqrt{c}^{-n}t_{1}^{-n}\sqrt{1-q^{-2}}^{-1}\cdots\sqrt{1-q^{-2n}}^{-1}%
\]
for all $n\in\mathbb{N}$, and $x_{1}$ is surjective since $x_{1}\equiv\sqrt
{c}t_{1}\alpha\equiv\sqrt{c}t_{1}\mathcal{S}^{\ast}$ modulo $\mathcal{K}$ is a
Fredholm operator of index $1$. So $x_{1}^{\ast}x_{2}$ is a Fredholm operator
of index $-2$. Actually, $\ker ( x_{1}^{\ast}x_{2} ) =0$ (hence
$ ( x_{1}^{\ast}x_{2} ) ^{\ast} ( x_{1}^{\ast}x_{2}) $
is invertible) and $\dim ( \operatorname{coker} ( x_{1}^{\ast}%
x_{2} ) ) =2$, and hence the partial isometry $(
x_{1}^{\ast}x_{2}) \vert x_{1}^{\ast}x_{2}\vert ^{-1}$ in
the polar decomposition of $x_{1}^{\ast}x_{2}$ is $\mathcal{S}\oplus
\mathcal{S}$ (up to a unitary direct summand) after a suitable choice of
orthonormal basis. This observation is consistent with our goal to show that
$C\big( \mathbb{C}P_{q,c}^{1}\big) $ is isomorphic to the pullback
$C^*$-algebra $\mathcal{T}\oplus_{C ( \mathbb{T} ) }\mathcal{T}$, but
is far from sufficient to make such a conclusion. We need to do a much more
detailed analysis which starts with the following computation.

First we compute
\begin{gather*}
x_{1}^{\ast}x_{1} =c\alpha^{\ast}\alpha+\sqrt{c}\overline{t_{1}}^{2}%
\alpha^{\ast}\gamma+\sqrt{c}t_{1}^{2}\gamma\alpha+\gamma^{2}\\
\hphantom{x_{1}^{\ast}x_{1}}{} =c+(1-c) \gamma^{2}+\sqrt{c}\overline{t_{1}}^{2}\alpha^{\ast
}\gamma+\sqrt{c}t_{1}^{2}\gamma\alpha\equiv c\quad \text{mod\ }\mathcal{K},
\\
x_{2}^{\ast}x_{2} =\alpha\alpha^{\ast}-q^{-1}\sqrt{c}t_{1}^{2}\alpha
\gamma-q^{-1}\sqrt{c}\overline{t_{1}}^{2}\gamma\alpha^{\ast}+q^{-2}c\gamma
^{2}\\
\hphantom{x_{2}^{\ast}x_{2}}{} =1+q^{-2}(c-1) \gamma^{2}-q^{-1}\sqrt{c}t_{1}^{2}\alpha
\gamma-q^{-1}\sqrt{c}\overline{t_{1}}^{2}\gamma\alpha^{\ast}\\
\hphantom{x_{2}^{\ast}x_{2}}{} =1+q^{-2}(c-1) \gamma^{2}-q^{-2}\sqrt{c}t_{1}^{2}\gamma
\alpha-q^{-2}\sqrt{c}\overline{t_{1}}^{2}\alpha^{\ast}\gamma\equiv
1\quad \text{mod\ }\mathcal{K},
\\
x_{1}^{\ast}x_{2}=\sqrt{c}\overline{t_{1}}^{2} ( \alpha^{\ast} )
^{2}-cq^{-1}\alpha^{\ast}\gamma+\gamma\alpha^{\ast}-q^{-1}\sqrt{c}t_{1}%
^{2}\gamma^{2}\equiv\sqrt{c}\overline{t_{1}}^{2}\mathcal{S}^{2} \quad \text{mod\ }
\mathcal{K},
\\
x_{2}^{\ast}x_{1}=\sqrt{c}t_{1}^{2}\alpha^{2}-cq^{-1}\gamma\alpha+\alpha
\gamma-q^{-1}\sqrt{c}\overline{t_{1}}^{2}\gamma^{2},
\end{gather*}
which imply
\[
x_{1}^{\ast}x_{1}+q^{2}x_{2}^{\ast}x_{2}=q^{2}+c.
\]
So the $C^*$-algebra $C\big( \mathbb{C}P_{q,c}^{1}\big) $ is generated by
$x_{1}^{\ast}x_{2}$ and $x_{1}^{\ast}x_{1}$, i.e.,
\[
C \big( \mathbb{C}P_{q,c}^{1} \big) =C^{\ast} ( \{ x_{1}^{\ast
}x_{2}, x_{1}^{\ast}x_{1} \} ) ,
\]
since $x_{2}^{\ast}x_{1}= ( x_{1}^{\ast}x_{2} ) ^{\ast}$ and
$x_{2}^{\ast}x_{2}=1+cq^{-2}-q^{-2}x_{1}^{\ast}x_{1}$ are generated by
$x_{1}^{\ast}x_{2}$ (with $ ( x_{1}^{\ast}x_{2} ) ^{\ast} (
x_{1}^{\ast}x_{2} ) $ invertible) and $x_{1}^{\ast}x_{1}$. As a remark,
we note that $x_{1}^{\ast}x_{1}$ and $x_{2}^{\ast}x_{2}$ commute.

We also note that
\[
x_{1}x_{1}^{\ast}+x_{2}x_{2}^{\ast}=1+c
\]
and hence $x_{1}x_{1}^{\ast}$ and $x_{2}x_{2}^{\ast}$ commute. Indeed
\begin{align*}
x_{1}x_{1}^{\ast} & =\big( \sqrt{c}t_{1}\alpha+t_{2}\gamma\big) \big(
\sqrt{c}t_{2}\alpha^{\ast}+t_{1}\gamma\big) =c\alpha\alpha^{\ast}+\sqrt
{c}t_{1}^{2}\alpha\gamma+\sqrt{c}t_{2}^{2}\gamma\alpha^{\ast}+\gamma^{2}\\
& =c-cq^{-2}\gamma^{2}+\sqrt{c}t_{1}^{2}q^{-1}\gamma\alpha+\sqrt{c}t_{2}
^{2}q^{-1}\alpha^{\ast}\gamma+\gamma^{2},
\end{align*}
while
\begin{align*}
x_{2}x_{2}^{\ast} & =\big( t_{2}\alpha^{\ast}-q^{-1}\sqrt{c}t_{1}
\gamma\big) \big( t_{1}\alpha-q^{-1}\sqrt{c}t_{2}\gamma\big)
=\alpha^{\ast}\alpha-q^{-1}\sqrt{c}t_{2}^{2}\alpha^{\ast}\gamma-q^{-1}\sqrt
{c}t_{1}^{2}\gamma\alpha+q^{-2}c\gamma^{2}\\
& =1-\gamma^{2}-q^{-1}\sqrt{c}t_{2}^{2}\alpha^{\ast}\gamma-q^{-1}\sqrt
{c}t_{1}^{2}\gamma\alpha+q^{-2}c\gamma^{2}.
\end{align*}

Before proceeding further, we recall some operator-theoretic properties often
used implicitly in the following analysis, including that $\overline
{\operatorname{range}(T) }=\ker ( T^{\ast} )
^{\perp}$ and $\ker (T^{\ast} ) =\ker ( T^{\ast}T ) $
for general bounded linear operators $T$ on a Hilbert space $\mathcal{H}$
easily derived from $ \langle T^{\ast}(v) ,w \rangle
= \langle v,T(w) \rangle $ and $ \langle (
T^{\ast}T ) (v) ,v \rangle = \langle T (
v ) ,T(v) \rangle $ for all $v,w\in\mathcal{H}$
respectively. In $C^*$-algebra theory, a projection refers to a self-adjoint
idempotent. For an operator $T$ in the $C^*$-algebra $\mathcal{B} (
\mathcal{H} ) $ of all bounded linear operators on $\mathcal{H}$, we
recall that $T$ is a projection if and only if $T$ is geometrically the
orthogonal projection from $\mathcal{H}$ onto a closed subspace of~$\mathcal{H}$.

We will need some basic knowledge of Fredholm operators, i.e., operators
$T\in\mathcal{B} ( \mathcal{H} ) $ with its quotient class $ [
T ] $ an invertible element of the Calkin algebra $\mathcal{B} (
\mathcal{H} ) /\mathcal{K} ( \mathcal{H} ) $, which can be
found in \cite{Do,Mu}. Any such operator has
closed finite-codimensional range and finite-dimensional kernel, and the
intersection of $\mathbb{R}\backslash \{ 0 \} $ and the spectrum
$\operatorname{Sp}(T) $ of any positive Fredholm operator $T$
is a compact subset of $ ( 0,\infty ) $. Also we note that the set
of all Fredholm operators is closed under taking adjoint and composition of operators.
For any positive Fredholm operator~$T$ we will denote by $\underline{T^{-1/2}}$
the positive
operator $f(T) $ defined by functional calculus, where $f$
is the nonnegative continuous function on $ \{ 0 \} \sqcup K$
such that $f ( \bullet ) =\bullet^{-1/2}$ on $K:=\operatorname{Sp}(T) \backslash \{ 0 \} $ and $f (
0 ) =0$.

Below we recall a folklore result with a proof.

\begin{lemma}\label{Fredholm} For any Fredholm operator $T$ on a Hilbert space~$\mathcal{H}$,
\[
\tilde{T}:=T\underline{( T^{\ast}T) ^{-1/2}}
\]
is a partial isometry sending the closed subspace $( \ker(T)) ^{\perp}\equiv\operatorname{range}( T^{\ast})
$ isometrically onto the closed subspace $\operatorname{range}(T) $ while annihilating $\ker(T) $.
\end{lemma}
\begin{proof}
Note that since $T^{\ast}T$ is a positive Fredholm operator, the set
$K$ is a compact subset of $ ( 0,\infty) $.

By the spectral theory of self-adjoint operators,
\[
\tilde{T}^{\ast}\tilde{T}\equiv\underline{( T^{\ast}T) ^{-1/2}%
}T^{\ast}T\underline{( T^{\ast}T ) ^{-1/2}}\equiv f (
T^{\ast}T ) ( T^{\ast}T ) f ( T^{\ast}T )
=\chi_{K} ( T^{\ast}T )
\]
for the characteristic function $\chi_{K}$ on $\operatorname{Sp} (
T^{\ast}T ) $, and hence $\tilde{T}^{\ast}\tilde{T}$ is the orthogonal
projection from~$\mathcal{H}$ onto $\operatorname{range} ( T^{\ast}T) $. Thus $\tilde{T}$ annihilates
\[
\ker\big(\tilde{T}\big) \equiv\ker\big( \tilde{T}^{\ast}\tilde
{T}\big) \equiv\big( \operatorname{range}\big( \tilde{T}^{\ast}%
\tilde{T}\big) \big) ^{\perp}= ( \operatorname{range} (
T^{\ast}T ) ) ^{\perp}\equiv\ker ( T^{\ast}T )
\equiv\ker ( T )
\]
and is metric preserving on
\[
\operatorname{range}( T^{\ast}T) \equiv( \ker(
T^{\ast}T) ) ^{\perp}\equiv( \ker(T)
) ^{\perp}\equiv\operatorname{range}( T^{\ast})
\]
due to
\[
\big\langle \tilde{T}(v) ,\tilde{T}(w)
\big\rangle =\big\langle \big( \tilde{T}^{\ast}\tilde{T}\big) (
v) ,w\big\rangle =\langle v,w\rangle \qquad \text{for all }v,w\in\operatorname{range}( T^{\ast}T) ,
\]
i.e., $\tilde{T}$ is a partial isometry sending $\operatorname{range} (
T^{\ast}T ) $ isometrically onto $\operatorname{range} ( \tilde
{T} ) $ while annihilating $\ker ( T ) $.

It remains to show that $\operatorname{range}(\tilde{T})
=\operatorname{range}(T) $. In fact, since $\frac{1}{f|_{K}}$
is a well-defined continuous function on $K$, the restriction of $ (
T^{\ast}T ) ^{-1/2}\equiv f ( T^{\ast}T ) $ to
$\operatorname{range}( T^{\ast}T) $ is an invertible linear
operator on $\operatorname{range} ( T^{\ast}T ) $ and hence
\[
\operatorname{range}\big( ( T^{\ast}T ) ^{-1/2}\big)
=\operatorname{range} ( T^{\ast}T ) \equiv ( \ker (T )) ^{\perp}.
\]
Thus we get
\[
\operatorname{range}(\tilde{T}) \equiv\operatorname{range}
\big( T ( T^{\ast}T ) ^{-1/2}\big) =T ( (
\ker(T) ) ^{\perp} ) =\operatorname{range} (T) .\tag*{\qed}
\]
\renewcommand{\qed}{}
\end{proof}

At each fixed $t_{1}\in\mathbb{T}$, by applying Lemma~\ref{Fredholm} to the Fredholm
operator values of the norm continuous $\mathbb{T}$-families $x_{1}$ and
$x_{2}$, we get two partial isometries
\[
\tilde{x}_{1}:=x_{1}\underline{( x_{1}^{\ast}x_{1}) ^{-1/2}}
\]
and
\[
\tilde{x}_{2}:=x_{2}\underline{( x_{2}^{\ast}x_{2}) ^{-1/2}%
}=x_{2}( x_{2}^{\ast}x_{2}) ^{-1/2}%
\]
where
\[
\underline{( x_{2}^{\ast}x_{2}) ^{-1/2}}=( x_{2}^{\ast
}x_{2}) ^{-1/2}\equiv\big( ( x_{2}^{\ast}x_{2})
^{-1}\big) ^{1/2}\equiv\sqrt{( x_{2}^{\ast}x_{2}) ^{-1}}%
\]
is a well-defined invertible operator on $\ell^{2}( \mathbb{Z}_{\geq
}) $ since $\ker( x_{2}) =0$ and hence the spectrum of the
positive Fredholm operator $x_{2}^{\ast}x_{2}$ is a compact subset of $(
0,\infty) $, implying the invertibility of $x_{2}^{\ast}x_{2}$ and
making the functional calculus $( x_{2}^{\ast}x_{2})
^{-1/2}\equiv\big( ( x_{2}^{\ast}x_{2}) ^{-1}\big) ^{1/2}$ meaningful.

\begin{theorem}\label{generators} The $C^*$-algebra $C\big(\mathbb{C}P_{q,c}^{1}\big)$ coincides with $C^{\ast} (
 \{ \tilde{x}_{1}^{\ast}\tilde{x}_{2},x_{1}^{\ast}x_{1} \} )
$, i.e.\ the $C^*$-algebra generated by two $\mathbb{T}$-families $x_{1}^{\ast}x_{1}\geq0$ and
$\tilde{x}_{1}^{\ast}\tilde{x}_{2}$ of operators on $\ell^{2} (
\mathbb{Z}_{\geq} ) $ such that at each fixed $t_{1}\mathbb{\in
}\mathbb{T}$, $\tilde{x}_{1}^{\ast}\tilde{x}_{2}$ is an isometry of index $-2$
$($with a zero kernel and a range of codimension~$2)$, where $\tilde{x}%
_{1}^{\ast}$ and $\tilde{x}_{2}=x_{2} ( x_{2}^{\ast}x_{2} )
^{-1/2}$ are isometries of index $-1$ while $\tilde{x}_{1}=x_{1}%
\underline{ ( x_{1}^{\ast}x_{1} ) ^{-1/2}}$ and $\tilde{x}%
_{2}^{\ast}$ are surjective partial isometries of index $1$.
\end{theorem}
\begin{proof} By Lemma \ref{Fredholm}, the surjective Fredholm operator $x_{1}$ with kernel of
dimension $1$ yields a surjective partial isometry $\tilde{x}_{1}%
=x_{1}\underline{( x_{1}^{\ast}x_{1}) ^{-1/2}}$ of index $1$ and
the injective Fredholm $x_{2}$ with cokernel of dimension $1$ yields an
isometry $\tilde{x}_{2}=x_{2}( x_{2}^{\ast}x_{2}) ^{-1/2}$ of
index~$-1$.

Now both $\tilde{x}_{2}$ and the adjoint $\tilde{x}_{1}^{\ast}$ of the
surjective partial isometry $\tilde{x}_{1}$ are isometries of index~$-1$, and
hence $\tilde{x}_{1}^{\ast}\tilde{x}_{2}$ is an isometry of index $-2$ with
$( \tilde{x}_{1}^{\ast}\tilde{x}_{2}) ^{\ast}( \tilde
{x}_{1}^{\ast}\tilde{x}_{2}) =1$.

From the definition of $\tilde{x}_{i}$, we get
\[
\tilde{x}_{1}^{\ast}\tilde{x}_{2}=\underline{( x_{1}^{\ast}x_{1})
^{-1/2}}x_{1}^{\ast}x_{2}( x_{2}^{\ast}x_{2}) ^{-1/2}\in C^{\ast
}( \{ x_{1}^{\ast}x_{2},x_{1}^{\ast}x_{1},x_{2}^{\ast}%
x_{2}\}) =C\big(\mathbb{C}P_{q,c}^{1}\big) .
\]
Since $( x_{1}^{\ast}x_{1}) ^{1/2}\underline{(
x_{1}^{\ast}x_{1}) ^{-1/2}}$ is the orthogonal projection onto
$\operatorname{range}( x_{1}^{\ast}x_{1}) \equiv
\operatorname{range}( x_{1}^{\ast}) $, by spectral theory:
\[
x_{1}^{\ast}x_{2}=( x_{1}^{\ast}x_{1}) ^{1/2}\underline{(
x_{1}^{\ast}x_{1}) ^{-1/2}}x_{1}^{\ast}x_{2}=( x_{1}^{\ast}%
x_{1}) ^{1/2}\tilde{x}_{1}^{\ast}\tilde{x}_{2}( x_{2}^{\ast}%
x_{2}) ^{1/2}\in C^{\ast}( \{ \tilde{x}_{1}^{\ast}\tilde
{x}_{2},x_{1}^{\ast}x_{1},x_{2}^{\ast}x_{2}\}) .
\]
So we get
\[
C\big(\mathbb{C}P_{q,c}^{1}\big) \equiv C^{\ast}(\{
x_{1}^{\ast}x_{2},x_{1}^{\ast}x_{1},x_{2}^{\ast}x_{2}\})
=C^{\ast}(\{ \tilde{x}_{1}^{\ast}\tilde{x}_{2},x_{1}^{\ast}%
x_{1},x_{2}^{\ast}x_{2}\} ) .
\]

Furthermore the generator $x_{2}^{\ast}x_{2}$ is redundant since%
\[
x_{2}^{\ast}x_{2}=1+cq^{-2}-q^{-2}x_{1}^{\ast}x_{1}=\big( 1+cq^{-2}\big)
 ( \tilde{x}_{1}^{\ast}\tilde{x}_{2} ) ^{\ast} ( \tilde{x}%
_{1}^{\ast}\tilde{x}_{2} ) -q^{-2}x_{1}^{\ast}x_{1}
\]
can be generated by $\tilde{x}_{1}^{\ast}\tilde{x}_{2}$ and $x_{1}^{\ast}%
x_{1}$. Thus
\[
C\big(\mathbb{C}P_{q,c}^{1}\big) \equiv C^{\ast} ( \{
\tilde{x}_{1}^{\ast}\tilde{x}_{2},x_{1}^{\ast}x_{1},x_{2}^{\ast}x_{2} \}
 ) =C^{\ast} ( \{ \tilde{x}_{1}^{\ast}\tilde{x}_{2},x_{1}^{\ast}x_{1} \} ) .\tag*{\qed}
\]
\renewcommand{\qed}{}
\end{proof}

We remark that for each $i\in \{ 1,2 \} $,
\[
\tilde{x}_{i}^{\ast}\tilde{x}_{i}=\underline{( x_{i}^{\ast}x_{i})
^{-1/2}}x_{i}^{\ast}x_{i}\underline{( x_{i}^{\ast}x_{i}) ^{-1/2}%
}=\chi_{\operatorname{Sp}( x_{i}^{\ast}x_{i}) \backslash\{
0\} }( x_{i}^{\ast}x_{i}) \in C^{\ast}( \{
x_{i}^{\ast}x_{i}\} )
\]
and hence belongs to $C\big(\mathbb{C}P_{q,c}^{1}\big) \equiv C^{\ast
} ( \{ \tilde{x}_{1}^{\ast}\tilde{x}_{2},x_{1}^{\ast}x_{1}%
,x_{2}^{\ast}x_{2} \} ) $, since any $C^*$-algebra is closed under
functional calculus by continuous functions vanishing at~$0$.

\section{Invariant subspace decomposition}

In this section, fixing an arbitrary value of the parameter $t_{1}%
\equiv\overline{t_{2}}\in\mathbb{T}$, we study and treat any $\mathbb{T}%
$-family of operators on $\ell^{2} ( \mathbb{Z}_{\geq} ) $,
including any element of $C\big(\mathbb{C}P_{q,c}^{1}\big) $, as an
operator in $\mathcal{B}\big( \ell^{2} ( \mathbb{Z}_{\geq} )\big) $.

As such the index-$1$ surjective partial isometry $\tilde{x}_{1}$ has
$\ker ( \tilde{x}_{1} ) =\ker ( x_{1} ) =\mathbb{C}%
v_{1}$ for some unit vector $v_{1}$, and
\[
p_{1}:=1-\tilde{x}_{1}^{\ast}\tilde{x}_{1}\in C^{\ast} ( \{
\tilde{x}_{1}^{\ast}\tilde{x}_{2},\tilde{x}_{1}^{\ast}\tilde{x}_{1} \}
 ) \subset C\big(\mathbb{C}P_{q,c}^{1}\big)
\]
is the rank-$1$ orthogonal projection onto $\mathbb{C}v_{1}$. Note that
$ ( \tilde{x}_{1}^{\ast}\tilde{x}_{1} ) ( v_{1} ) =0$
is equivalent to $\tilde{x}_{1} ( v_{1} ) =0$ (or equivalently
$v_{1}\perp\operatorname{range} ( \tilde{x}_{1}^{\ast} )
=\operatorname{range} ( \tilde{x}_{1}^{\ast}\tilde{x}_{1} ) $).

Note that
\[
p_{2}:=\tilde{x}_{1}^{\ast}\tilde{x}_{1}- ( \tilde{x}_{2}^{\ast}\tilde
{x}_{1} ) ^{\ast} ( \tilde{x}_{2}^{\ast}\tilde{x}_{1} )
=\tilde{x}_{1}^{\ast}\tilde{x}_{1}- ( \tilde{x}_{1}^{\ast}\tilde{x}%
_{2} ) ( \tilde{x}_{2}^{\ast}\tilde{x}_{1} ) =\tilde{x}%
_{1}^{\ast} ( 1-\tilde{x}_{2}\tilde{x}_{2}^{\ast} ) \tilde{x}_{1}%
\]
is also a rank-1 projection onto $\mathbb{C}v_{2}$ for some unit vector
$v_{2}$. In fact $p_{2}$ clearly annihilates $\ker ( \tilde{x}_{1} )
$ and can be viewed as the conjugation of the rank-$1$ projection $1-\tilde{x}_{2}\tilde{x}_{2}%
^{\ast}$ (onto the kernel of $\tilde{x}_{2}^{\ast}$) by the unitary
operator $\tilde{x}_{1}|_{( \ker( \tilde{x}_{1}))
^{\perp}}$, from $( \ker( \tilde{x}_{1})) ^{\perp}$
onto $\ell^{2}( \mathbb{Z}_{\geq}) $. In an explicit description, $v_{2}$
can be taken as the inverse image under $\tilde{x}_{1}|_{( \ker( \tilde{x}_{1})) ^{\perp}}$,
of any unit vector in the $1$-dimensional range of $1-\tilde{x}_{2}\tilde{x}%
_{2}^{\ast}$, and, in particular, $v_{2}\in( \ker( \tilde{x}
_{1})) ^{\perp}\equiv\operatorname{range}( \tilde{x}_{1}^{\ast}\tilde{x}_{1})$.
The inequalities
\[
0\leq p_{2}=\tilde{x}_{1}^{\ast}\tilde{x}_{1}-( \tilde{x}_{1}^{\ast
}\tilde{x}_{2}) ( \tilde{x}_{2}^{\ast}\tilde{x}_{1})
\leq\tilde{x}_{1}^{\ast}\tilde{x}_{1}
\]
relate the three projections $p_{2}$, $( \tilde{x}_{1}^{\ast}
\tilde{x}_{2}) ( \tilde{x}_{2}^{\ast}\tilde{x}_{1}) $, and
$\tilde{x}_{1}^{\ast}\tilde{x}_{1}$ in $C^{\ast}(\{ \tilde{x}
_{1}^{\ast}\tilde{x}_{2},\tilde{x}_{1}^{\ast}\tilde{x}_{1}\})
$, and clarify their geometric relation: $p_{2}$ and $( \tilde
{x}_{1}^{\ast}\tilde{x}_{2}) ( \tilde{x}_{2}^{\ast}\tilde{x}%
_{1}) $ are projections onto two mutually orthogonal subspaces which
add up to the range of the projection $\tilde{x}_{1}^{\ast}\tilde{x}_{1}$,
i.e.,
\[
\operatorname{range}( p_{2}) \oplus^{\perp}\operatorname{range}%
(( \tilde{x}_{1}^{\ast}\tilde{x}_{2}) ( \tilde
{x}_{2}^{\ast}\tilde{x}_{1})) =\operatorname{range}(
\tilde{x}_{1}^{\ast}\tilde{x}_{1}) ,
\]
indicating, in particular, $v_{2}\in\operatorname{range}( p_{2})
\subset\operatorname{range}( \tilde{x}_{1}^{\ast}\tilde{x}_{1})$.

Now $v_{2}\perp v_{1}$ since $v_{2}$ is in the range of the self-adjoint
operator $\tilde{x}_{1}^{\ast}\tilde{x}_{1}$ and hence is perpendicular to
$\ker( \tilde{x}_{1}^{\ast}\tilde{x}_{1}) =\mathbb{C}v_{1}$.
Furthermore,
\[
v_{i+2n}:=( \tilde{x}_{1}^{\ast}\tilde{x}_{2}) ^{n}(v_{i})
\]
with $n\geq0$ and $i\in\{ 1,2\} $ are orthonormal vectors, since
\[
v_{1}\perp\operatorname{range}( \tilde{x}_{1}^{\ast})
\supset\operatorname{range}( \tilde{x}_{1}^{\ast}\tilde{x}_{2})
\ni( \tilde{x}_{1}^{\ast}\tilde{x}_{2}) ( v_{1})
\]
and
\[
v_{2}\perp\operatorname{range}(( \tilde{x}_{1}^{\ast}\tilde
{x}_{2}) ( \tilde{x}_{2}^{\ast}\tilde{x}_{1}))
=\operatorname{range}( \tilde{x}_{1}^{\ast}\tilde{x}_{2})
\qquad \text{with\ \ }v_{1}\perp v_{2}\, .
\]
Thus $\mathcal{V}:=\operatorname{Span}\{ v_{1},v_{2}\}
\perp\operatorname{range}( \tilde{x}_{1}^{\ast}\tilde{x}_{2}) $
or more precisely, by combining with the fact that the index-$(
-2) $ isometry $\tilde{x}_{1}^{\ast}\tilde{x}_{2}$ has
$\operatorname{range}( \tilde{x}_{1}^{\ast}\tilde{x}_{2}) $ of
codimension $2$,
\[
\mathcal{H}=\mathcal{V~}\oplus^{\perp}\operatorname{range}( \tilde
{x}_{1}^{\ast}\tilde{x}_{2}) \qquad \text{as Hilbert space orthogonal
direct sum.}
\]
Hence, since $\tilde{x}_{1}^{\ast}\tilde{x}_{2}$ is an isometry, we inductively get:
\[
\operatorname{range}( \tilde{x}_{1}^{\ast}\tilde{x}_{2})
^{k-1}=( \tilde{x}_{1}^{\ast}\tilde{x}_{2}) ^{k-1}(
\mathcal{V}) \oplus^{\perp}\operatorname{range}( \tilde{x}_{1}^{\ast}\tilde{x}_{2}) ^{k}
\]
for all $k\geq1$.

Clearly, for each $i\in\{ 1,2\} $, the operator $\tilde{x}
_{1}^{\ast}\tilde{x}_{2}$ restricted to the closed linear span $\mathcal{H}
_{i}\subset\ell^{2}( \mathbb{Z}_{\geq}) $ of $\{
v_{i+2n}\colon n\geq0\} $ is a unilateral shift $\mathcal{S}$, while the
orthogonal projection onto $\mathbb{C}v_{i}$ is $p_{i}\in C^{\ast}(
\{ \tilde{x}_{1}^{\ast}\tilde{x}_{2},\tilde{x}_{1}^{\ast}\tilde{x}%
_{1}\})$.

Since $\tilde{x}_{1}^{\ast}\tilde{x}_{2}$ is a unilateral shift simultaneously
on both $\mathcal{H}_{1}$ and $\mathcal{H}_{2}$, it generates a $C^*$-algebra
$C^{\ast}(\{ \tilde{x}_{1}^{\ast}\tilde{x}_{2}\}
|_{\mathcal{H}_{1}\oplus\mathcal{H}_{2}}) $ containing two
``synchronized'' copies of the ideal of
compact operators, i.e.,
\[
C^{\ast}( \{ \tilde{x}_{1}^{\ast}\tilde{x}_{2}\}
|_{\mathcal{H}_{1}\oplus\mathcal{H}_{2}}) \supset\big\{ T\oplus
T\colon T\in\mathcal{K}\big( \ell^{2}( \mathbb{Z}_{\geq}) \big)
\big\} \cong\mathcal{K}\big( \ell^{2}( \mathbb{Z}_{\geq})\big),
\]
where $\mathcal{H}_{1}$ and $\mathcal{H}_{2}$ are identified with the same
Hilbert space $\ell^{2}( \mathbb{Z}_{\geq}) $ in a canonical way,
i.e., identifying $v_{i+2n}$ for $i\in\{ 1,2\} $ with the
canonical orthonormal basis vector $e_{n}\in\ell^{2}( \mathbb{Z}_{\geq
}) $.

However our goal is to show that $C\big(\mathbb{C}P_{q,c}^{1}\big) $ or
for now $C^{\ast} ( \{ \tilde{x}_{1}^{\ast}\tilde{x}_{2},\tilde
{x}_{1}^{\ast}\tilde{x}_{1} \} |_{\mathcal{H}_{1}\oplus\mathcal{H}_{2}
} ) $ contains the\ direct sum $\mathcal{K} ( \mathcal{H}
_{1} ) \oplus\mathcal{K}( \mathcal{H}_{2} ) $ of all
``non-synchronized'' pairs of compact
operators. This can be achieved by noticing that for any $k,m\in
\mathbb{Z}_{\geq}$,
\[
\varepsilon_{k,m}^{( 1) }:=( \tilde{x}_{1}^{\ast}\tilde
{x}_{2}) ^{k}p_{1}( ( \tilde{x}_{1}^{\ast}\tilde{x}%
_{2}) ^{\ast}) ^{m}|_{\mathcal{H}_{1}\oplus\mathcal{H}_{2}}\in
C^{\ast}( \{ \tilde{x}_{1}^{\ast}\tilde{x}_{2},\tilde{x}_{1}%
^{\ast}\tilde{x}_{1}\} |_{\mathcal{H}_{1}\oplus\mathcal{H}_{2}})
\]
is a typical matrix unit in $\mathcal{K}( \mathcal{H}_{1})
\oplus0$ sending $v_{1+2m}$ to $v_{1+2k}$ while eliminating all other
$v_{i+2n}$ with $i+2n\neq1+2m$, and we get $\mathcal{K}( \mathcal{H}%
_{1}) \oplus0$ as the closure of the linear span of $\varepsilon
_{k,m}^{( 1) }$ with $k,m\in\mathbb{Z}_{\geq}$. Similarly the
elements
\[
\varepsilon_{k,m}^{(2) }:=( \tilde{x}_{1}^{\ast}\tilde
{x}_{2}) ^{k}p_{2}( ( \tilde{x}_{1}^{\ast}\tilde{x}%
_{2}) ^{\ast}) ^{m}|_{\mathcal{H}_{1}\oplus\mathcal{H}_{2}}\in
C^{\ast}( \{ \tilde{x}_{1}^{\ast}\tilde{x}_{2},\tilde{x}_{1}%
^{\ast}\tilde{x}_{1}\} |_{\mathcal{H}_{1}\oplus\mathcal{H}_{2}})
\]
with $k,m\in\mathbb{Z}_{\geq}$ linearly span a dense subspace of
$0\oplus\mathcal{K}( \mathcal{H}_{2}) $. Thus
\[
\mathcal{K}( \mathcal{H}_{1}) \oplus\mathcal{K}(
\mathcal{H}_{2}) =( \mathcal{K}( \mathcal{H}_{1})
\oplus0) +( 0\oplus\mathcal{K}( \mathcal{H}_{2})
) \subset C^{\ast}( \{ \tilde{x}_{1}^{\ast}\tilde{x}%
_{2},\tilde{x}_{1}^{\ast}\tilde{x}_{1}\} |_{\mathcal{H}_{1}%
\oplus\mathcal{H}_{2}}) .
\]

Next we want to show that each $v_{k}$ with $k\geq1$ is an eigenvector of
$x_{1}^{\ast}x_{1}$ or equivalently of $x_{2}^{\ast}x_{2}=1+cq^{-2}%
-q^{-2}x_{1}^{\ast}x_{1}$, and hence each $\mathcal{H}_{i}$ is invariant under
$x_{1}^{\ast}x_{1}$ and $x_{2}^{\ast}x_{2}$.

\begin{proposition}\label{intertwine} The isometry $\tilde{x}_{1}^{\ast}\tilde{x}_{2}$
intertwines the positive operators $x_{1}^{\ast}x_{1}$ and $(
1+c) -x_{2}^{\ast}x_{2}$, i.e.,
\[
( x_{1}^{\ast}x_{1}) ( \tilde{x}_{1}^{\ast}\tilde{x}%
_{2}) =( \tilde{x}_{1}^{\ast}\tilde{x}_{2}) (
1+c-x_{2}^{\ast}x_{2}) .
\]
\end{proposition}

\begin{proof}A direct computation shows
\begin{gather*}
( x_{1}^{\ast}x_{1}) ( \tilde{x}_{1}^{\ast}\tilde{x}_{2}) =x_{1}^{\ast}x_{1}\underline{( x_{1}^{\ast}x_{1})
^{-1/2}}x_{1}^{\ast}\tilde{x}_{2}=\underline{( x_{1}^{\ast}x_{1})
^{-1/2}}x_{1}^{\ast}x_{1}x_{1}^{\ast}\tilde{x}_{2}%
\\
\hphantom{( x_{1}^{\ast}x_{1}) ( \tilde{x}_{1}^{\ast}\tilde{x}_{2})}{}
=\underline{( x_{1}^{\ast}x_{1}) ^{-1/2}}x_{1}^{\ast}(
1+c-x_{2}x_{2}^{\ast}) \tilde{x}_{2}=\tilde{x}_{1}^{\ast}(
1+c-x_{2}x_{2}^{\ast}) \tilde{x}_{2}\\
\hphantom{( x_{1}^{\ast}x_{1}) ( \tilde{x}_{1}^{\ast}\tilde{x}_{2})}{}
=( 1+c) \tilde{x}%
_{1}^{\ast}\tilde{x}_{2}-\tilde{x}_{1}^{\ast}x_{2}x_{2}^{\ast}\tilde{x}_{2}
=( 1+c) \tilde{x}_{1}^{\ast}\tilde{x}_{2}-\tilde{x}_{1}^{\ast
}x_{2}x_{2}^{\ast}x_{2}( x_{2}^{\ast}x_{2}) ^{-1/2}\\
\hphantom{( x_{1}^{\ast}x_{1}) ( \tilde{x}_{1}^{\ast}\tilde{x}_{2})}{}
=(1+c) \tilde{x}_{1}^{\ast}\tilde{x}_{2}-\tilde{x}_{1}^{\ast}x_{2}(
x_{2}^{\ast}x_{2}) ^{-1/2}x_{2}^{\ast}x_{2}
=( 1+c) \tilde{x}_{1}^{\ast}\tilde{x}_{2}-( \tilde{x}%
_{1}^{\ast}\tilde{x}_{2}) ( x_{2}^{\ast}x_{2}) \\
\hphantom{( x_{1}^{\ast}x_{1}) ( \tilde{x}_{1}^{\ast}\tilde{x}_{2})}{}
=(\tilde{x}_{1}^{\ast}\tilde{x}_{2}) ( 1+c-x_{2}^{\ast}x_{2}) .\tag*{\qed}
\end{gather*}
\renewcommand{\qed}{}
\end{proof}

\begin{proposition}\label{eigen-intertwine} The isometry $\tilde{x}_{1}^{\ast}\tilde{x}_{2}$
intertwines the $($possibly degenerate$)$ eigenspaces $E_{\lambda}(
x_{2}^{\ast}x_{2}) $ and $E_{1+c-\lambda}( x_{1}^{\ast}%
x_{1}) $, where $E_{\lambda}(T) :=\ker(\lambda-T) $ for linear operators $T$ and $\lambda\in\mathbb{C}$. More
precisely,
\[
( \tilde{x}_{1}^{\ast}\tilde{x}_{2}) ( E_{\lambda}(
x_{2}^{\ast}x_{2}) ) \subset E_{1+c-\lambda}( x_{1}^{\ast}x_{1}) ,
\]
and
\[
( \tilde{x}_{1}^{\ast}\tilde{x}_{2}) ^{-1}( E_{1+c-\lambda
}( x_{1}^{\ast}x_{1}) ) \subset( E_{\lambda}(x_{2}^{\ast}x_{2}) ),
\]
where $( \tilde{x}_{1}^{\ast}\tilde{x}_{2}) ^{-1}(
E_{1+c-\lambda}( x_{1}^{\ast}x_{1}) ) $ is the inverse
image of $E_{1+c-\lambda}( x_{1}^{\ast}x_{1}) $ under $($the
non-surjective$)$ $\tilde{x}_{1}^{\ast}\tilde{x}_{2}$.
\end{proposition}
\begin{proof} The commutation relation
\[
( x_{1}^{\ast}x_{1}) ( \tilde{x}_{1}^{\ast}\tilde{x}%
_{2}) =( \tilde{x}_{1}^{\ast}\tilde{x}_{2}) (
( 1+c) -x_{2}^{\ast}x_{2})
\]
implies that if $v\in E_{\lambda}( x_{2}^{\ast}x_{2}) $ then
$( \tilde{x}_{1}^{\ast}\tilde{x}_{2}) (v) \in
E_{1+c-\lambda}( x_{1}^{\ast}x_{1}) $, because
\begin{gather*}
( x_{1}^{\ast}x_{1}) ( ( \tilde{x}_{1}^{\ast}\tilde{x}_{2}) (v) ) =( \tilde{x}_{1}^{\ast
}\tilde{x}_{2}) ( ( 1+c) -x_{2}^{\ast}x_{2})(v)\\
\hphantom{( x_{1}^{\ast}x_{1}) ( ( \tilde{x}_{1}^{\ast}\tilde{x}_{2}) (v) )}{}
=( \tilde{x}_{1}^{\ast}\tilde{x}_{2}) ( ( 1+c)
-\lambda) v=( ( 1+c) -\lambda) (( \tilde{x}_{1}^{\ast}\tilde{x}_{2}) (v) ).
\end{gather*}

On the other hand, if $( \tilde{x}_{1}^{\ast}\tilde{x}_{2})(v) \in E_{1+c-\lambda}( x_{1}^{\ast}x_{1}) $, then
\begin{gather*}
( ( 1+c) -\lambda) ( ( \tilde{x}_{1}^{\ast}\tilde{x}_{2}) (v) ) =(
x_{1}^{\ast}x_{1}) ( ( \tilde{x}_{1}^{\ast}\tilde{x}%
_{2}) (v) ) =( \tilde{x}_{1}^{\ast}\tilde
{x}_{2}) ( ( 1+c) -x_{2}^{\ast}x_{2}) (v)\\
\hphantom{( ( 1+c) -\lambda) ( ( \tilde{x}_{1}^{\ast}\tilde{x}_{2}) (v) )}{}
=( 1+c) ( \tilde{x}_{1}^{\ast}\tilde{x}_{2}) (
v) -( \tilde{x}_{1}^{\ast}\tilde{x}_{2}) ( (
x_{2}^{\ast}x_{2}) (v) ),
\end{gather*}
and hence $( \tilde{x}_{1}^{\ast}\tilde{x}_{2}) ( \lambda
v) =( \tilde{x}_{1}^{\ast}\tilde{x}_{2}) ( (x_{2}^{\ast}x_{2}) (v) ) $. Since $\tilde{x}_{1}^{\ast}\tilde{x}_{2}$ is injective, we get $\lambda v=( x_{2}^{\ast}x_{2}) (v) $, i.e., $v\in E_{\lambda}( x_{2}^{\ast}x_{2}) $.
\end{proof}

\begin{corollary}\label{eigenvalue} If $\lambda$ is an eigenvalue of $x_{2}^{\ast}x_{2}$,
then $1+c-\lambda$ is an eigenvalue of $x_{1}^{\ast}x_{1}$.
\end{corollary}
\begin{proof} If $E_{\lambda}( x_{2}^{\ast}x_{2}) \neq0$ then
$E_{1+c-\lambda}( x_{1}^{\ast}x_{1}) \supset( \tilde{x}%
_{1}^{\ast}\tilde{x}_{2}) ( E_{\lambda}( x_{2}^{\ast}%
x_{2}) ) \neq0$ since $\tilde{x}_{1}^{\ast}\tilde{x}_{2}$ is injective.
\end{proof}

The equality $x_{1}^{\ast}x_{1}+q^{2}x_{2}^{\ast}x_{2}=q^{2}+c$ implies
\[
E_{\lambda}( x_{2}^{\ast}x_{2}) =E_{q^{2}+c-q^{2}\lambda}(
x_{1}^{\ast}x_{1})
\]
for any $\lambda\in\mathbb{R}$, or equivalently
\[
E_{\lambda}( x_{1}^{\ast}x_{1}) =E_{q^{-2}( q^{2}%
+c-\lambda) }( x_{2}^{\ast}x_{2}) .
\]

\begin{proposition}\label{recursive} The orthonormal vectors $v_{k}$, $k\in\mathbb{N}$, are
eigenvectors of $x_{1}^{\ast}x_{1}$ $($and of $x_{2}^{\ast}x_{2}\equiv
1+q^{-2}c-q^{-2}x_{1}^{\ast}x_{1})$, and hence each of $\mathcal{H}_{1}$ and
$\mathcal{H}_{2}$ is invariant under all of the generators $\tilde{x}
_{1}^{\ast}\tilde{x}_{2}$, $x_{1}^{\ast}x_{1}$, and $x_{2}^{\ast}x_{2}$ of
$C\big(\mathbb{C}P_{q,c}^{1}\big) $. More explicitly, $(x_{1}^{\ast}x_{1}) ( v_{k}) =c_{k}v_{k}$ for all
$k\in\mathbb{N}$, where $c_{k}$ is defined recursively by
\[
c_{k+2}=c-q^{-2}c+q^{-2}c_{k}\qquad \text{for }k\in\mathbb{N},\text{\ with\ } c_{2}=1+c\text{\ and }c_{1}=0,
\]
which can be rewritten as
\[
c_{2n}=q^{-2( n-1) }+c\qquad \text{and} \qquad c_{2n+1}=\big(1-q^{-2n}\big) c
\]
for all $n\in\mathbb{Z}_{\geq}$.
\end{proposition}
\begin{proof} We prove $( x_{1}^{\ast}x_{1}) ( v_{k})=c_{k}v_{k}$ and the formula $c_{k+2}=c-q^{-2}c+q^{-2}c_{k}$ inductively on
$k$.

First $v_{1}\in( \operatorname{range}( \tilde{x}_{1}^{\ast
}) ) ^{\perp}=\ker( \tilde{x}_{1}) $, so
\[
( x_{1}^{\ast}x_{1}) ( v_{1}) =\big( (
x_{1}^{\ast}x_{1}) ^{1/2}x_{1}^{\ast}x_{1}\underline{(
x_{1}^{\ast}x_{1}) ^{-1/2}}\big) ( v_{1}) =\big(
( x_{1}^{\ast}x_{1}) ^{1/2}x_{1}^{\ast}\tilde{x}_{1}\big)( v_{1}) =0.
\]

Next since $v_{2}\in\operatorname{range}( \tilde{x}_{1}^{\ast}(1-\tilde{x}_{2}\tilde{x}_{2}^{\ast}) \tilde{x}_{1}) $, so
$v_{2}=\tilde{x}_{1}^{\ast}(w) $ for some unit vector
\[
w\in\operatorname{range}( 1-\tilde{x}_{2}\tilde{x}_{2}^{\ast})
=\ker( \tilde{x}_{2}\tilde{x}_{2}^{\ast}) =\ker( \tilde
{x}_{2}^{\ast}) =\ker( x_{2}^{\ast})
\]
and hence
\begin{gather*}
( x_{1}^{\ast}x_{1}) ( v_{2}) =( x_{1}^{\ast
}x_{1}) ( \tilde{x}_{1}^{\ast}(w) ) =(
x_{1}^{\ast}x_{1}) \underline{( x_{1}^{\ast}x_{1})
^{-1/2}}x_{1}^{\ast}(w)
\\
\hphantom{( x_{1}^{\ast}x_{1}) ( v_{2})}{}
=\underline{( x_{1}^{\ast}x_{1}) ^{-1/2}}( x_{1}^{\ast
}x_{1}) x_{1}^{\ast}(w) =\underline{( x_{1}^{\ast
}x_{1}) ^{-1/2}}x_{1}^{\ast}( 1+c-x_{2}x_{2}^{\ast})
(w)\\
\hphantom{( x_{1}^{\ast}x_{1}) ( v_{2})}{}
=\tilde{x}_{1}^{\ast}( ( 1+c) w-0) =(
1+c) \tilde{x}_{1}^{\ast}(w) =( 1+c) v_{2}.
\end{gather*}

Now assume that $( x_{1}^{\ast}x_{1}) ( v_{k})
=c_{k}v_{k}$, i.e., $v_{k}\in E_{c_{k}}( x_{1}^{\ast}x_{1}) $, for
$k\in\mathbb{N}$. Then
\begin{gather*}
v_{k+2}\equiv( \tilde{x}_{1}^{\ast}\tilde{x}_{2}) (
v_{k}) \in( \tilde{x}_{1}^{\ast}\tilde{x}_{2}) (
E_{c_{k}}( x_{1}^{\ast}x_{1}) ) \equiv( \tilde
{x}_{1}^{\ast}\tilde{x}_{2}) ( E_{q^{-2}( q^{2}+c-c_{k}) }( x_{2}^{\ast}x_{2}) )
\\
\hphantom{v_{k+2}}{}
\subset E_{1+c-q^{-2}( q^{2}+c-c_{k}) }( x_{1}^{\ast}%
x_{1}) =E_{c-q^{-2}c+q^{-2}c_{k}}( x_{1}^{\ast}x_{1}) ,
\end{gather*}
and hence $( x_{1}^{\ast}x_{1}) ( v_{k+2})
=c_{k+2}v_{k+2}$ for $c_{k+2}:=c-q^{-2}c+q^{-2}c_{k}$.

The recursive formula $c_{k+2}=c-q^{-2}c+q^{-2}c_{k}$ rewritten as
$c_{k+2}-c=q^{-2}( c_{k}-c) $ immediately leads to $c_{i+2n}
-c=q^{-2n}( c_{i}-c) $ and hence
\[
c_{i+2n}=q^{-2n}( c_{i}-c) +c
\]
for any $i\in\{ 1,2\} $ and $n\in\mathbb{N}$. More explicitly, we
have $c_{2n}=q^{-2( n-1) }+c$ and $c_{2n+1}=\big(1-q^{-2n}\big) c$ for all $n\in\mathbb{N}$.
\end{proof}

\begin{corollary}\label{weight_shift} The element $x_{1}^{\ast}x_{2}=( x_{1}^{\ast}x_{1})
^{1/2}\tilde{x}_{1}^{\ast}\tilde{x}_{2}( x_{2}^{\ast}x_{2})
^{1/2}$ is a weighted shift on $\mathcal{H}_{1}$ and $\mathcal{H}_{2}$ with
respect to the orthonormal bases $\{ v_{2n-1}\} _{n\geq1}$ and
$\{ v_{2n}\} _{n\geq1}$ respectively. More precisely,
\[
( x_{1}^{\ast}x_{2}) ( v_{k}) =\sqrt{c_{k+2}}%
\sqrt{1+q^{-2}c-q^{-2}c_{k}}v_{k+2}%
\]
for the constants $c_{k}$ specified in the above proposition.
\end{corollary}
\begin{proof} This is a simple consequence of $( x_{1}^{\ast}x_{1})
( v_{k}) =c_{k}v_{k}$ and
\[
( x_{2}^{\ast}x_{2}) ( v_{k}) \equiv\big(
1+q^{-2}c-q^{-2}x_{1}^{\ast}x_{1}\big) ( v_{k}) =\big(
1+q^{-2}c-q^{-2}c_{k}\big) v_{k} .\tag*{\qed}
\]\renewcommand{\qed}{}
\end{proof}

With each of $\mathcal{H}_{1}$ and $\mathcal{H}_{2}$ invariant under the
self-adjoint operators $x_{i}^{\ast}x_{i}$, it is clear that the orthogonal
complement $\mathcal{H}_{0}:=( \mathcal{H}_{1}\oplus\mathcal{H}%
_{2}) ^{\perp}$ in $\ell^{2}( \mathbb{Z}_{\geq}) $ is also
invariant under each $x_{i}^{\ast}x_{i}$. On the other hand, since we know the
orthonormal vectors $v_{1},v_{2}\in( \operatorname{range}(
\tilde{x}_{1}^{\ast}\tilde{x}_{2}) ) ^{\perp}$ for the
index-$( -2) $ isometry $\tilde{x}_{1}^{\ast}\tilde{x}_{2}$, we
get a Wold-von Neumann decomposition (Theorem 3.5.17 of \cite{Mu}) for the
isometry $\tilde{x}_{1}^{\ast}\tilde{x}_{2}$ as
\[
\tilde{x}_{1}^{\ast}\tilde{x}_{2}=\tilde{x}_{1}^{\ast}\tilde{x}_{2}%
|_{\mathcal{H}_{0}}\oplus\tilde{x}_{1}^{\ast}\tilde{x}_{2}|_{\mathcal{H}_{1}%
}\oplus\tilde{x}_{1}^{\ast}\tilde{x}_{2}|_{\mathcal{H}_{2}}%
\]
with
\[
\mathcal{H}_{0}\equiv\big( \operatorname{Span}\big( \big\{ (
\tilde{x}_{1}^{\ast}\tilde{x}_{2}) ^{k}( v_{i})
\colon i\in\{ 1,2\} \text{ and }k\in\mathbb{Z}_{\geq}\big\} \big)
\big) ^{\perp} .
\]
Here $( \tilde{x}_{1}^{\ast}\tilde{x}_{2}) |_{\mathcal{H}_{0}}$
is a unitary operator on $\mathcal{H}_{0}$ (if $\mathcal{H}_{0}\neq0$) since
$\tilde{x}_{1}^{\ast}\tilde{x}_{2}|_{\mathcal{H}_{0}}$ is an index-$0$
isometry in view of $\tilde{x}_{1}^{\ast}\tilde{x}_{2}|_{\mathcal{H}_{i}}$
being an index-$( -1) $ isometry for each $i\in\{
1,2\} $.

It is not clear whether $\mathcal{H}_{0}$ is actually trivial or not, so we
remark that any discussion involving $\mathcal{H}_{0}$ below is only needed
and valid when $\mathcal{H}_{0}\neq0$.

We already know that $( \tilde{x}_{1}^{\ast}\tilde{x}_{2})
|_{\mathcal{H}_{i}}$ is a unilateral shift for each $i\in\{ 1,2\}
$. So with respect to the decomposition
\[
\ell^{2}( \mathbb{Z}_{\geq}) =\mathcal{H}_{0}\oplus
\mathcal{H}_{1}\oplus\mathcal{H}_{2}%
\]
into orthogonal subspaces, the generators $\tilde{x}_{1}^{\ast}\tilde{x}_{2}$,
$x_{1}^{\ast}x_{1}$, $x_{2}^{\ast}x_{2}$ and hence all elements of $C\big(
\mathbb{C}P_{q,c}^{1}\big) $ can be viewed as block diagonal operators.
Then it is easy to see that Propositions~\ref{intertwine} and~\ref{eigen-intertwine} hold for the
restrictions of $\tilde{x}_{1}^{\ast}\tilde{x}_{2}$, $x_{1}^{\ast}x_{1}$,
$x_{2}^{\ast}x_{2}$ to each $\mathcal{H}_{i}$.

\begin{lemma}\label{invariant-spectrum} The spectrum $\operatorname{Sp}( x_{1}^{\ast}%
x_{1}|_{\mathcal{H}_{0}}) $ of $x_{1}^{\ast}x_{1}|_{\mathcal{H}_{0}}$
is invariant under the function
\[
f_{1}\colon \ s\mapsto c-q^{-2}c+q^{-2}s\equiv c+q^{-2}(s-c)
\]
and its inverse function. Similarly, the spectrum $\operatorname{Sp}(
x_{2}^{\ast}x_{2}|_{\mathcal{H}_{0}}) $ of $x_{2}^{\ast}x_{2}%
|_{\mathcal{H}_{0}}$ is invariant under the function%
\[
f_{2}\colon \ s\mapsto1-q^{-2}+q^{-2}s\equiv1+q^{-2}(s-1)
\]
and its inverse function.
\end{lemma}
\begin{proof} By Proposition \ref{intertwine},
\[
( x_{1}^{\ast}x_{1}|_{\mathcal{H}_{0}}) ( \tilde{x}%
_{1}^{\ast}\tilde{x}_{2}|_{\mathcal{H}_{0}}) =( \tilde{x}%
_{1}^{\ast}\tilde{x}_{2}|_{\mathcal{H}_{0}}) ( 1+c-x_{2}^{\ast
}x_{2}|_{\mathcal{H}_{0}})
\]
with $\tilde{x}_{1}^{\ast}\tilde{x}_{2}|_{\mathcal{H}_{0}}$ unitary, we get
$x_{1}^{\ast}x_{1}|_{\mathcal{H}_{0}}$ and $1+c-x_{2}^{\ast}x_{2}%
|_{\mathcal{H}_{0}}$ unitarily equivalent and hence
\[
\operatorname{Sp}( x_{1}^{\ast}x_{1}|_{\mathcal{H}_{0}})
=\operatorname{Sp}( 1+c-x_{2}^{\ast}x_{2}|_{\mathcal{H}_{0}})
=1+c-\operatorname{Sp}( x_{2}^{\ast}x_{2}|_{\mathcal{H}_{0}}) .
\]

On the other hand, from $x_{1}^{\ast}x_{1}+q^{2}x_{2}^{\ast}x_{2}=q^{2}+c$, we
have
\[
\operatorname{Sp}( x_{2}^{\ast}x_{2}|_{\mathcal{H}_{0}})
=q^{-2}\big( q^{2}+c-\operatorname{Sp}( x_{1}^{\ast}x_{1}
|_{\mathcal{H}_{0}}) \big) =1+q^{-2}c-q^{-2}\operatorname{Sp}
( x_{1}^{\ast}x_{1}|_{\mathcal{H}_{0}}) .
\]
Hence
\[
\operatorname{Sp}( x_{1}^{\ast}x_{1}|_{\mathcal{H}_{0}})
=1+c-\big( 1+q^{-2}c-q^{-2}\operatorname{Sp}( x_{1}^{\ast}
x_{1}|_{\mathcal{H}_{0}}) \big) =c-q^{-2}c+q^{-2}\operatorname{Sp}
( x_{1}^{\ast}x_{1}|_{\mathcal{H}_{0}}) ,
\]
which shows that under the invertible function $f_{1}\colon s\in\mathbb{R}\mapsto
c-q^{-2}c+q^{-2}s\in\mathbb{R}$, the set $\operatorname{Sp}(
x_{1}^{\ast}x_{1}|_{\mathcal{H}_{0}}) \subset\mathbb{R}$ equals itself
and hence the inverse function $( f_{1}) ^{-1}$ maps
$\operatorname{Sp}( x_{1}^{\ast}x_{1}|_{\mathcal{H}_{0}}) $ onto
itself too.

Since the invertible function $g\colon s\in\mathbb{R}\mapsto1+c-s\in\mathbb{R}$ maps
$\operatorname{Sp}( x_{2}^{\ast}x_{2}|_{\mathcal{H}_{0}}) $ onto
$\operatorname{Sp}( x_{1}^{\ast}x_{1}|_{\mathcal{H}_{0}}) $, the
conjugate $g^{-1}\circ f_{1}\circ g$ and its inverse function map
$\operatorname{Sp}( x_{2}^{\ast}x_{2}|_{\mathcal{H}_{0}}) $ onto
itself, where
\begin{gather*}
\big( g^{-1}\circ f_{1}\circ g\big) (s) =1+c-f_{1}(1+c-s)\\
\hphantom{\big( g^{-1}\circ f_{1}\circ g\big) (s)}{}
=1+c-\big( c-q^{-2}c+q^{-2} ( 1+c-s ) \big) =1-q^{-2}%
+q^{-2}s.\tag*{\qed}
\end{gather*}\renewcommand{\qed}{}
\end{proof}

Note that $f_{1}(s) -c=q^{-2}(s-c) $ and
$f_{2}(s) -1=q^{-2}(s-1) $ for all $s\in
\mathbb{R}$ with $q>1$. So the only bounded backward $f_{1}$-orbit is the
constant $f_{1}$-orbit $\{c\} $, and similarly the only bounded
backward $f_{2}$-orbit is the constant $f_{2}$-orbit $\{1\} $,
where by a backward $f_{i}$-orbit, we mean the set $ \{ (
f_{i} ) ^{-n}(s) \colon n\in\mathbb{N} \} $ for a point
$s\in\mathbb{R}$. On the other hand, any forward $f_{i}$-orbit converges to
the constant $f_{i}$-orbit, i.e., $\lim\limits_{n\rightarrow\infty}(
f_{1}) ^{n}(s) =c$ and $\lim\limits_{n\rightarrow\infty}(
f_{2}) ^{n}(s) =1$ for any $s$. Since each spectrum $\operatorname{Sp}( x_{i}^{\ast}%
x_{i}|_{\mathcal{H}_{0}}) $ is a compact and hence bounded set that is
invariant under backward iterations of~$f_{i}$, it can contain only the
constant $f_{i}$-orbit. We have, therefore

\begin{corollary}\label{H0spectrum} $\operatorname{Sp}( x_{1}^{\ast}x_{1}%
|_{\mathcal{H}_{0}}) =\{c\} $ and $\operatorname{Sp}
( x_{2}^{\ast}x_{2}|_{\mathcal{H}_{0}}) =\{1\} $,
i.e., $x_{1}^{\ast}x_{1}|_{\mathcal{H}_{0}}=c\, \mathrm{id}$ and $x_{2}^{\ast
}x_{2}|_{\mathcal{H}_{0}}=\mathrm{id}$.
\end{corollary}

\begin{proposition}\label{iso} There is a unital $C^*$-algebra isomorphism from
$C^{\ast}( \{ \tilde{x}_{1}^{\ast}\tilde{x}_{2},x_{1}^{\ast}%
x_{1} \} ) |_{\mathcal{H}_{1}\oplus\mathcal{H}_{2}}$ to the
pullback $\mathcal{T}\oplus_{C(\mathbb{T}) }\mathcal{T}$ of two
copies of $\mathcal{T}\overset{\sigma}{\rightarrow}C(\mathbb{T})
$, sending $\tilde{x}_{1}^{\ast}\tilde{x}_{2}|_{\mathcal{H}_{1}\oplus
\mathcal{H}_{2}}$ to $\mathcal{S}\oplus\mathcal{S}$. This isomorphism provides
an exact sequence%
\[
0\rightarrow\mathcal{K}( \mathcal{H}_{1}) \oplus\mathcal{K}%
( \mathcal{H}_{2}) \rightarrow C^{\ast}(\{ \tilde
{x}_{1}^{\ast}\tilde{x}_{2},x_{1}^{\ast}x_{1}\} ) |_{\mathcal{H}%
_{1}\oplus\mathcal{H}_{2}}\cong\mathcal{T}\oplus_{C(\mathbb{T})
}\mathcal{T}\overset{\sigma}{\rightarrow}C(\mathbb{T})
\rightarrow0
\]
of $C^*$-algebras with $\sigma( \tilde{x}_{1}^{\ast}\tilde{x}%
_{2}|_{\mathcal{H}_{1}\oplus\mathcal{H}_{2}}) =\mathrm{id}_{\mathbb{T}%
}$, $\sigma( x_{1}^{\ast}x_{1}|_{\mathcal{H}_{1}\oplus\mathcal{H}_{2}%
}) =c$, and $\sigma( x_{2}^{\ast}x_{2}|_{\mathcal{H}_{1}%
\oplus\mathcal{H}_{2}}) =1$.
\end{proposition}
\begin{proof} We note that the eigenvalues $c_{k}$ of $x_{1}^{\ast}x_{1}%
|_{\mathcal{H}_{1}\oplus\mathcal{H}_{2}}$ satisfying $c_{k+2}=c-q^{-2}%
c+q^{-2}c_{k}=f_{1}( c_{k}) $ form two forward $f_{1}$-orbits and
hence $\lim\limits_{k\rightarrow\infty}c_{k}=c$. This limit is also clear from the
explicit formulae of $c_{2n}$ and $c_{2n+1}$ given in Proposition~\ref{recursive}.
Similarly, one can verify that the eigenvalues~$c_{k}^{\prime}$ of
$x_{2}^{\ast}x_{2}|_{\mathcal{H}_{1}\oplus\mathcal{H}_{2}}$ form two forward
$f_{2}$-orbit and hence $\lim\limits_{k\rightarrow\infty}c_{k}^{\prime}=1$.

So $x_{1}^{\ast}x_{1}|_{\mathcal{H}_{1}\oplus\mathcal{H}_{2}}\equiv c\oplus c$
mod $\mathcal{K}( \mathcal{H}_{1}) \oplus\mathcal{K}(
\mathcal{H}_{2}) $ and $x_{2}^{\ast}x_{2}|_{\mathcal{H}_{1}%
\oplus\mathcal{H}_{2}}\equiv1\oplus1$ mod $\mathcal{K}( \mathcal{H}%
_{1}) \oplus\mathcal{K}( \mathcal{H}_{2}) $, while
$\tilde{x}_{1}^{\ast}\tilde{x}_{2}|_{\mathcal{H}_{1}\oplus\mathcal{H}_{2}%
}=\mathcal{S}_{\mathcal{H}_{1}}\oplus\mathcal{S}_{\mathcal{H}_{2}}$ for copies
$\mathcal{S}_{\mathcal{H}_{i}}$ of the unilateral shift.

It has been shown earlier that $\mathcal{K}( \mathcal{H}_{1})
\oplus\mathcal{K}( \mathcal{H}_{2}) \subset C^{\ast}(
\{ \tilde{x}_{1}^{\ast}\tilde{x}_{2},x_{1}^{\ast}x_{1}\} )
|_{\mathcal{H}_{1}\oplus\mathcal{H}_{2}}$, so it is not hard to see that
$C^{\ast}( \{ \tilde{x}_{1}^{\ast}\tilde{x}_{2},x_{1}^{\ast}%
x_{1} \} ) |_{\mathcal{H}_{1}\oplus\mathcal{H}_{2}}$ is the
pullback of two copies of $\mathcal{T}\overset{\sigma}{\rightarrow}C(
\mathbb{T}) $. In fact, $\mathcal{S}_{\mathcal{H}_{1}}\oplus
\mathcal{S}_{\mathcal{H}_{2}}\equiv\tilde{x}_{1}^{\ast}\tilde{x}%
_{2}|_{\mathcal{H}_{1}\oplus\mathcal{H}_{2}}$ generates $ \{ T\oplus
T\colon T\in\mathcal{T} \} $ as a $C^*$-subalgebra of $C^{\ast}( \{
\tilde{x}_{1}^{\ast}\tilde{x}_{2},x_{1}^{\ast}x_{1} \} )
|_{\mathcal{H}_{1}\oplus\mathcal{H}_{2}}$ and hence%
\[
\tilde{x}_{1}^{\ast}\tilde{x}_{2}|_{\mathcal{H}_{1}\oplus\mathcal{H}_{2}}%
\in \{ T\oplus T\colon T\in\mathcal{T} \} +( \mathcal{K}(
\mathcal{H}_{1}) \oplus\mathcal{K}( \mathcal{H}_{2})
) \subset C^{\ast}( \{ \tilde{x}_{1}^{\ast}\tilde{x}%
_{2},x_{1}^{\ast}x_{1} \} ) |_{\mathcal{H}_{1}\oplus
\mathcal{H}_{2}}.
\]
On the other hand,
\[
x_{1}^{\ast}x_{1}|_{\mathcal{H}_{1}\oplus\mathcal{H}_{2}}\in( c\oplus
c) +( \mathcal{K}( \mathcal{H}_{1}) \oplus
\mathcal{K}( \mathcal{H}_{2}) ) \subset \{ T\oplus
T\colon T\in\mathcal{T} \} +( \mathcal{K}( \mathcal{H}_{1})
\oplus\mathcal{K}( \mathcal{H}_{2}) )
\]
and hence
\[
C^{\ast}( \{ \tilde{x}_{1}^{\ast}\tilde{x}_{2},x_{1}^{\ast}
x_{1} \} ) |_{\mathcal{H}_{1}\oplus\mathcal{H}_{2}}
\subset \{ T\oplus T\colon T\in\mathcal{T} \} +( \mathcal{K}(
\mathcal{H}_{1}) \oplus\mathcal{K}( \mathcal{H}_{2})) .
\]
So we get
\[
C^{\ast}( \{ \tilde{x}_{1}^{\ast}\tilde{x}_{2},x_{1}^{\ast}%
x_{1} \} ) |_{\mathcal{H}_{1}\oplus\mathcal{H}_{2}}= \{
T\oplus T\colon T\in\mathcal{T} \} +( \mathcal{K}( \mathcal{H}%
_{1}) \oplus\mathcal{K}( \mathcal{H}_{2}) )
=\mathcal{T}\oplus_{C(\mathbb{T}) }\mathcal{T},
\]
where the second equality is due to that any $S\oplus T\in\mathcal{T}%
\oplus\mathcal{T}$ with $\sigma( S) =\sigma(T) $
can be written as
\[
S\oplus T=( T\oplus T) +( ( S-T)
\oplus0) \in T\oplus T+( \mathcal{K}( \mathcal{H}
_{1}) \oplus\mathcal{K}( \mathcal{H}_{2}) ) .
\]

Replacing $\mathcal{T}\oplus_{C(\mathbb{T}) }\mathcal{T}$ in the
canonical exact sequence
\[
0\rightarrow\mathcal{K}( \mathcal{H}_{1}) \oplus\mathcal{K}%
( \mathcal{H}_{2}) \rightarrow\mathcal{T}\oplus_{C(
\mathbb{T}) }\mathcal{T}\overset{\sigma}{\rightarrow}C(
\mathbb{T}) \rightarrow0
\]
by the isomorphic $C^*$-algebra $C^{\ast}( \{ \tilde{x}_{1}^{\ast
}\tilde{x}_{2},x_{1}^{\ast}x_{1} \} ) |_{\mathcal{H}_{1}%
\oplus\mathcal{H}_{2}}$, we get the stated exact sequence with $\sigma(
\tilde{x}_{1}^{\ast}\tilde{x}_{2}|_{\mathcal{H}_{1}\oplus\mathcal{H}_{2}%
}) =\mathrm{id}_{\mathbb{T}}$, $\sigma( x_{1}^{\ast}%
x_{1}|_{\mathcal{H}_{1}\oplus\mathcal{H}_{2}}) =c$, and $\sigma(
x_{2}^{\ast}x_{2}|_{\mathcal{H}_{1}\oplus\mathcal{H}_{2}}) =1$.
\end{proof}

\begin{theorem}\label{restriction} The restriction map
\[
T\in C^{\ast}( \{ \tilde{x}_{1}^{\ast}\tilde{x}_{2},x_{1}^{\ast
}x_{1} \} ) \mapsto T|_{\mathcal{H}_{1}\oplus\mathcal{H}_{2}}\in
C^{\ast}( \{ \tilde{x}_{1}^{\ast}\tilde{x}_{2},x_{1}^{\ast}%
x_{1} \} ) |_{\mathcal{H}_{1}\oplus\mathcal{H}_{2}}%
\]
is a $C^*$-algebra isomorphism, and hence $C^{\ast}( \{ \tilde{x}%
_{1}^{\ast}\tilde{x}_{2},x_{1}^{\ast}x_{1} \} ) $ is isomorphic
to the pullback $\mathcal{T}\oplus_{C(\mathbb{T}) }\mathcal{T}$
of two copies of $\mathcal{T}\overset{\sigma}{\rightarrow}C(
\mathbb{T}) $ with $\tilde{x}_{1}^{\ast}\tilde{x}_{2}$ corresponding to
$\mathcal{S}\oplus\mathcal{S}$.
\end{theorem}
\begin{proof} Clearly we only need to consider the case with $\mathcal{H}_{0}\neq0$.

Since $\tilde{x}_{1}^{\ast}\tilde{x}_{2}|_{\mathcal{H}_{0}}$ is unitary, as
shown in the above discussion of Wold--von Neumann decomposition, and $C(
\mathbb{T}) $ is the universal $C^*$-algebra generated by a single unitary
generator, there is a unique $C^*$-algebra homomorphism
\[
h\colon \ C(\mathbb{T}) \rightarrow C^{\ast}( \{ \tilde
{x}_{1}^{\ast}\tilde{x}_{2}|_{\mathcal{H}_{0}} \} )
\]
sending $\mathrm{id}_{\mathbb{T}}$ to $\tilde{x}_{1}^{\ast}\tilde{x}%
_{2}|_{\mathcal{H}_{0}}$ while fixing all scalars in $\mathbb{C}\subset
C(\mathbb{T}) $.

Clearly with $x_{1}^{\ast}x_{1}|_{\mathcal{H}_{0}}=1$ and $x_{2}^{\ast}%
x_{2}|_{\mathcal{H}_{0}}=c$,
\[
h\circ\sigma\colon \ C^{\ast}( \{ \tilde{x}_{1}^{\ast}\tilde{x}_{2}%
,x_{1}^{\ast}x_{1} \} ) |_{\mathcal{H}_{1}\oplus\mathcal{H}_{2}%
}\rightarrow C^{\ast}( \{ \tilde{x}_{1}^{\ast}\tilde{x}%
_{2}|_{\mathcal{H}_{0}} \} ) =C^{\ast}( \{ \tilde
{x}_{1}^{\ast}\tilde{x}_{2},x_{1}^{\ast}x_{1} \} ) |_{\mathcal{H}_{0}}
\]
is a well-defined $C^*$-algebra homomorphism sending $\tilde{x}_{1}^{\ast}%
\tilde{x}_{2}|_{\mathcal{H}_{1}\oplus\mathcal{H}_{2}}$ to $\tilde{x}_{1}%
^{\ast}\tilde{x}_{2}|_{\mathcal{H}_{0}}$ and $x_{i}^{\ast}x_{i}|_{\mathcal{H}%
_{1}\oplus\mathcal{H}_{2}}$ to $x_{i}^{\ast}x_{i}|_{\mathcal{H}_{0}}$ for
$i\in \{ 1,2 \} $. Hence the restriction map
\[
T\in C^{\ast}( \{ \tilde{x}_{1}^{\ast}\tilde{x}_{2},x_{1}^{\ast
}x_{1} \} ) \mapsto T|_{\mathcal{H}_{1}\oplus\mathcal{H}_{2}}\in
C^{\ast}( \{ \tilde{x}_{1}^{\ast}\tilde{x}_{2},x_{1}^{\ast}%
x_{1} \} ) |_{\mathcal{H}_{1}\oplus\mathcal{H}_{2}}
\]
gives a well-defined isomorphism.
\end{proof}

In Theorem \ref{restriction}, we treat elements of $C^{\ast} ( \{ \tilde{x}%
_{1}^{\ast}\tilde{x}_{2},x_{1}^{\ast}x_{1} \} ) $ as operators
instead of families of operators by fixing implicitly the value of
$\mathbb{T}$-parameter at any $t_{1}\in\mathbb{T}$, i.e., the statement of
Theorem \ref{restriction} is a pointwise result at any $t_{1}\in\mathbb{T}$. It is clear that
collectively the restriction map
\[
C\big(\mathbb{C}P_{q,c}^{1}\big) \rightarrow C\big( \mathbb{C}
P_{q,c}^{1}\big) \big|_{\widetilde{\mathcal{H}_{1}}\oplus\widetilde{\mathcal{H}_{2}}}
\]
is still a $C^*$-algebra isomorphism where elements of $C\big( \mathbb{C}P_{q,c}^{1}\big) $ are $\mathbb{T}$-families of operators on $\ell^{2} ( \mathbb{Z}_{\geq} ) $ and $\widetilde{\mathcal{H}_{1}}\oplus\widetilde{\mathcal{H}_{2}}$ represents a $\mathbb{T}$-family of Hilbert subspaces $\mathcal{H}_{1}\oplus\mathcal{H}_{2}$ of $\ell^{2} (\mathbb{Z}_{\geq} ) $ constructed pointwise for each $t_{1}\in \mathbb{T}$ as described above.

\section{Superfluous circle parameter}

In this section, we show that the $\mathbb{T}$-parameter is superfluous for
the $C^*$-algebra $C\big( \mathbb{C}P_{q,c}^{1}\big)
\big|_{\widetilde{\mathcal{H}_{1}}\oplus\widetilde{\mathcal{H}_{2}}}$ consisting
of $\mathbb{T}$-families of operators on $\mathcal{H}_{1}\oplus\mathcal{H}%
_{2}$, and hence $C\big(\mathbb{C}P_{q,c}^{1}\big)
|_{\widetilde{\mathcal{H}_{1}}\oplus\widetilde{\mathcal{H}_{2}}}\cong C\big(
\mathbb{C}P_{q,c}^{1}\big) $ is isomorphic to $C^{\ast} ( \{
\tilde{x}_{1}^{\ast}\tilde{x}_{2},x_{1}^{\ast}x_{1} \} )
|_{\mathcal{H}_{1}\oplus\mathcal{H}_{2}}\cong\mathcal{T}\oplus_{C (
\mathbb{T} ) }\mathcal{T}$ (for any $t_{1}\in\mathbb{T}$ fixed) as
obtained in Proposition~\ref{iso}.

Recall that by a simple change of orthonormal basis $e_{k}\rightsquigarrow
t^{k}e_{k}$ of $\ell^{2} ( \mathbb{Z}_{\geq} ) $ for any fixed
$t\in\mathbb{T}$, the weighted shift operator $\alpha$ becomes $\tilde{\alpha
}=t\alpha$ with respect to the new orthonormal basis, while the self-adjoint
operator $\gamma$ remains the same operator $\tilde{\gamma}=\gamma$.

Note that the earlier concrete description of $x_{1}^{\ast}x_{2}$,
$x_{1}^{\ast}x_{1}$, and $x_{2}^{\ast}x_{2}$ as families of operators
parametrized by $t_{1}\in\mathbb{T}$ (with $t_{2}=\overline{t_{1}}$) viewed as
a representation of $C\big(\mathbb{C}P_{q,c}^{1}\big) \equiv C^{\ast
} ( \{ x_{1}^{\ast}x_{2},x_{1}^{\ast}x_{1},\allowbreak x_{2}^{\ast}
x_{2} \} ) $ can be first ``consolidated'' by a change of orthonormal basis converting
$\alpha$ to $\tilde{\alpha}:=t_{1}^{2}\alpha$ and $\gamma$ to $\tilde{\gamma
}=\gamma$, so that we can rewrite the description as
\begin{gather*}
x_{1}^{\ast}x_{1} =c+(1-c) \gamma^{2}+\sqrt{c}%
\overline{t_{1}}^{2}\alpha^{\ast}\gamma+\sqrt{c}t_{1}^{2}\gamma\alpha\\
\hphantom{x_{1}^{\ast}x_{1}}{} =c+(1-c) \tilde{\gamma}^{2}+\sqrt{c}\tilde{\alpha}^{\ast
}\tilde{\gamma}+\sqrt{c}\tilde{\gamma}\tilde{\alpha},
\\
x_{2}^{\ast}x_{2} =1+q^{-2}(c-1) \gamma^{2}-q^{-2}\sqrt
{c}t_{1}^{2}\gamma\alpha-q^{-2}\sqrt{c}\overline{t_{1}}^{2}\alpha^{\ast}%
\gamma\\
\hphantom{x_{2}^{\ast}x_{2}}{} =1+q^{-2}(c-1) \tilde{\gamma}^{2}-q^{-2}\sqrt{c}\tilde
{\gamma}\tilde{\alpha}-q^{-2}\sqrt{c}\tilde{\alpha}^{\ast}\tilde{\gamma},
\\
x_{1}^{\ast}x_{2} =\sqrt{c}\overline{t_{1}}^{2} ( \alpha^{\ast
} ) ^{2}-cq^{-1}\alpha^{\ast}\gamma+\gamma\alpha^{\ast}-q^{-1}\sqrt
{c}t_{1}^{2}\gamma^{2}\\
\hphantom{x_{1}^{\ast}x_{2}}{}
 =t_{1}^{2}\big( \sqrt{c}\overline{t_{1}}^{4} ( \alpha^{\ast} )
^{2}-cq^{-1}\overline{t_{1}}^{2}\alpha^{\ast}\gamma+\overline{t_{1}}^{2}%
\gamma\alpha^{\ast}-q^{-1}\sqrt{c}\gamma^{2}\big) \\
\hphantom{x_{1}^{\ast}x_{2}}{} =t_{1}^{2}\big( \sqrt{c} ( \tilde{\alpha}^{\ast} )
^{2}-cq^{-1}\tilde{\alpha}^{\ast}\tilde{\gamma}+\tilde{\gamma}\tilde{\alpha
}^{\ast}-q^{-1}\sqrt{c}\tilde{\gamma}^{2}\big) ,
\end{gather*}
where it is understood that $\tilde{\alpha},\tilde{\gamma}$ with respect to
suitable orthonormal basis of $\ell^{2} ( \mathbb{Z}_{\geq} ) $ are
the same familiar matrix operators $\alpha,\gamma$, and hence we can simply
replace $\tilde{\alpha},\tilde{\gamma}$ by $\alpha,\gamma$ in the above
formulas for $x_{1}^{\ast}x_{2}$, $x_{1}^{\ast}x_{1}$, and $x_{2}^{\ast}x_{2}$.

So we have
\begin{gather*}
x_{1}^{\ast}x_{1} =c+(1-c) \gamma^{2}+\sqrt{c}\alpha^{\ast
}\gamma+\sqrt{c}\gamma\alpha,\\
x_{2}^{\ast}x_{2} =1+q^{-2}(c-1) \gamma^{2}-q^{-2}\sqrt
{c}\gamma\alpha-q^{-2}\sqrt{c}\alpha^{\ast}\gamma,\\
x_{1}^{\ast}x_{2} =t_{1}^{2}\big( \sqrt{c} ( \alpha^{\ast} )
^{2}-cq^{-1}\alpha^{\ast}\gamma+\gamma\alpha^{\ast}-q^{-1}\sqrt{c}\gamma
^{2}\big) ,
\end{gather*}
where only $x_{1}^{\ast}x_{2}$ still involves $t_{1}=\overline{t_{2}}$ as a
factor. From Proposition~\ref{recursive} and Corollary~\ref{weight_shift}, there is an orthonormal
basis $ \{ v_{2k},v_{2k-1}\colon k\in\mathbb{N} \} $ of $\mathcal{H}%
_{1}\oplus\mathcal{H}_{2}$ consisting of eigenvectors of $x_{1}^{\ast}x_{1}$
and $x_{2}^{\ast}x_{2}$ and with respect to which $x_{1}^{\ast}x_{2}$ is a
double weighted shift. So after the change of orthonormal basis $v_{2k}%
\rightsquigarrow\big( t_{1}^{2}\big) ^{k}v_{2k}$ and $v_{2k-1}%
\rightsquigarrow\big( t_{1}^{2}\big) ^{k}v_{2k-1}$, the factor $t_{1}%
^{2}$ in the formula of $x_{1}^{\ast}x_{2}$ can be dropped while the formulas
of $x_{1}^{\ast}x_{1}$ and $x_{2}^{\ast}x_{2}$ remain the same, i.e., we have
$C^{\ast} ( \{ x_{1}^{\ast}x_{2},x_{1}^{\ast}x_{1} \}
 ) |_{\mathcal{H}_{1}\oplus\mathcal{H}_{2}}$ for any fixed $t_{1}%
\in\mathbb{T}$ unitarily equivalent to $C^{\ast} ( \{ x_{1}^{\ast
}x_{2},x_{1}^{\ast}x_{1} \} ) |_{\mathcal{H}_{1}\oplus
\mathcal{H}_{2}}$ for $t_{1}:=1$. So we conclude that the parameter $t_{1}%
\in\mathbb{T}$ is ``edundant'' in the sense
that representations of the generators $x_{1}^{\ast}x_{2}$, $x_{1}^{\ast}%
x_{1}$, and $x_{2}^{\ast}x_{2}$ of $C\big( \mathbb{C}P_{q,c}^{1}\big) $
as operators on $\ell^{2} ( \mathbb{Z}_{\geq} ) $ by the above
formulas for different $t_{1}$'s in $\mathbb{T}$ are unitarily equivalent representations.

So we can now say that $C\big( \mathbb{C}P_{q,c}^{1}\big) \cong C^{\ast
}( \{ x_{1}^{\ast}x_{2},x_{1}^{\ast}x_{1}\} ) $
where $C^{\ast}( \{ x_{1}^{\ast}x_{2},x_{1}^{\ast}x_{1}\}
) $ is considered as in the previous section for the operators
$x_{1}^{\ast}x_{2},x_{1}^{\ast}x_{1}$ without specifying any value of the
$t_{1}$-parameter. Thus we conclude that $C\big( \mathbb{C}P_{q,c}^{1}\big) $ is isomorphic to the pullback $\mathcal{T}\oplus_{C(\mathbb{T}) }\mathcal{T}$ of two copies of $\mathcal{T}\overset{\sigma
}{\rightarrow}C(\mathbb{T}) $ by Proposition~\ref{iso}, and hence is
isomorphic to the algebra $C\big( \mathbb{S}_{\mu c}^{2}\big) $ of
Podle\'{s} quantum 2-sphere by the result of~\cite{Sh:qp}.

We now summarize our conclusion in the following theorem, where the operators
$X_{1}:=\sqrt{c}\alpha+\gamma$ and $X_{2}:=-q^{-1}\sqrt{c}\gamma+\alpha^{\ast
}$ on $\ell^{2} ( \mathbb{Z}_{\geq} ) $ are respectively the values
of the $\mathbb{T}$-families $x_{1}$ and $x_{2}$ at $t_{1}=1=t_{2}$.

\begin{theorem}\label{final} The $C^*$-algebra $C\big( \mathbb{C}P_{q,c}^{1}\big)
\cong C^{\ast} ( \{ X_{1}^{\ast}X_{2},X_{1}^{\ast}X_{1} \}
 ) $ for the linear operators $X_{1}:=\sqrt{c}\alpha+\gamma$ and
$X_{2}:=-q^{-1}\sqrt{c}\gamma+\alpha^{\ast}$ on $\ell^{2} (
\mathbb{Z}_{\geq} ) $, and is isomorphic to the pullback
\[
\mathcal{T}\oplus_{C ( \mathbb{T} ) }\mathcal{T}\equiv\big\{
 ( T,S ) \in\mathcal{T}\oplus\mathcal{T}\colon \sigma(T)
=\sigma ( S ) \big\}
\]
of two copies of the standard Toeplitz $C^*$-algebra $\mathcal{T}$ along the
symbol map $\mathcal{T}\overset{\sigma}{\rightarrow}C( \mathbb{T})$.
\end{theorem}

We remark that the above change of orthonormal basis $v_{2k}\rightsquigarrow
\big( t_{1}^{2}\big) ^{k}v_{2k}$ and $v_{2k-1}\rightsquigarrow\big(
t_{1}^{2}\big) ^{k}v_{2k-1}$ is ``compatible'' and hence works well with the elements
$x_{1}^{\ast}x_{2}$, $x_{1}^{\ast}x_{1}$, and $x_{2}^{\ast}x_{2}$ of $C\big(
\mathbb{C}P_{q,c}^{1}\big) $, but is not suitable for manipulating more
fundamental elements like $x_{1}$ and $x_{2}$ in $C\big( \mathbb{S}_{q}^{3}\big) \equiv C( {\rm SU}_{q}(2)) $.

\subsection*{Acknowledgements}

N.~Ciccoli was partially supported by INDAM-GNSAGA and Fondo Ricerca di
Base 2017 ``Geometria della quantizzazione''.
A.J.-L.~Sheu was partially supported by University of Perugia~-- Visiting Researcher
Program, the grant H2020-MSCA-RISE-2015-691246-QUANTUM DYNAMICS, and the
Polish government grant 3542/H2020/2016/2.

\pdfbookmark[1]{References}{ref}
\LastPageEnding

\end{document}